\nonstopmode

\documentclass[11pt]{article} 

\usepackage{cite}
\usepackage{amsmath}
\usepackage{amsthm}
\usepackage{amsfonts}
\usepackage{amssymb}
\usepackage{mathtools}
\usepackage{mathrsfs}
\usepackage{enumerate}
\usepackage{graphicx}
\usepackage{caption}
\usepackage{subcaption}

\newtheorem{thm}{Theorem}
\newtheorem{lem}[thm]{Lemma}
\newtheorem{prop}[thm]{Proposition}
\newtheorem{coro}[thm]{Corollary}
\theoremstyle{definition}

\theoremstyle{remark}
\newtheorem{rmk}[thm]{Remark}

\DeclareMathOperator{\sgn}{sgn}

\DeclareMathOperator{\supp}{supp}
\DeclareMathOperator{\diam}{diam}
\DeclareMathOperator{\dist}{dist}
\DeclareMathOperator{\divergence}{div}
\DeclareMathOperator{\proj}{proj}
\DeclareMathOperator{\vol}{vol}

\DeclareMathOperator{\image}{image}

\newcommand{\R}{\mathbb{R}}
\newcommand{\Z}{\mathbb{Z}}

\newcommand{\Cinfc}{{C^\infty_c(T^nM)}}
\newcommand{\Mn}{{\mathscr M_n}}
\newcommand{\mass}{{\mathbf M}}

\newcommand{\N}{\mathbb{N}}

\newcommand{\Pn}{{\mathscr{V}_n}}

\newcommand{\cyccondition}{{(Cyc)}}
\newcommand{\holcondition}{{(Hol)}}

\newcommand{\measure}[1]{{\wr{#1}\wr}}
\newcommand{\base}{{\bar\mu}}
\newcommand{\lebesgue}{{\mathsf{m}}}
\newcommand{\constisoperimetric}{{K}}
\newcommand{\baseapprox}{{\beta}}

\newcommand{\finalchain}{\eta_k}

\newcommand{\remainder}{\zeta}

\title{Polygonal approximations of closed parametric varifolds}
\author{Rodolfo R\'ios-Zertuche}

\begin{document}

\maketitle

\begin{center}
\emph{A mi t\'ia Angelines y a nuestra amiga Simone}
\end{center}

\begin{abstract}
 We define holonomic measures to be certain analogues of varifolds that keep track of the local parameterization and orientation of the submanifold they represent. They are Borel measures on the direct sum of several copies of the tangent bundle. 
 
 We show that there is an approximation to these by smooth singular chains whose boundaries and Lagrangian actions are controlled. 
 
 As an illustration of the usefulness of this result, we show how this can be applied to study foliations on the torus. We give other applications elsewhere.
\end{abstract}

\tableofcontents

\section{Introduction}
\label{sec:intro}
%

Given a smooth manifold $M$ with tangent bundle $TM$, and a curve \[\gamma\colon [0,T]\to M,\] 
consider the measure $\mu_\gamma$ on $TM$ induced by $\gamma$ by pushing forward the Lebesgue measure on $[0,T]$ under the map $d\gamma\colon [0,T]\to TM$. In other words, $\mu_\gamma$ is given by
\[\int_{TM} f\,d\mu_\gamma=\int_0^Tf(d\gamma(t))\,dt\]
for measurable $f\colon TM\to\R$. The measure $\mu_\gamma$ is known as the \emph{Young measure} associated to $\gamma$. It is a very useful object in the calculus of variations, and it has thus been studied extensively (see for example \cite{manhe,matheractionminimizing91,bangert,patrick,fathibook} and the references therein). In particular, it appears as the main subject of study in the Mather theory for Lagrangian systems \cite{matheractionminimizing91}.

In this paper we consider an $n$-dimensional generalization of this concept. Our measures will be certain Borel measures on the direct sum $T^nM$ of $n$ copies of the tangent bundle $TM$ of a smooth manifold $M$. As such, they are analogous to varifolds \cite{allard,almgrenvarifolds} because they are measures that induce currents, but contain more information as they keep track of the local parameterization and orientation, so one could refer to holonomic measures as a kind of ``parameterized varifold.'' The idea is to have a framework for the study of minimizers of anisotropic Lagrangians with no \emph{a priori} symmetries. See \cite[Section 1.2]{myvariationalstructure} for examples of such Lagranginas.

Note that since one can also consider the differential forms $\omega$ on $M$ as functions on $T^nM$, our measures $\mu$ induce normal currents $T_\mu$ given by 
\[\langle T_\mu,\omega\rangle=\int_{T^nM}\omega\,d\mu.\]

We distinguish two classes of measures. First, those for which the integrals of exact forms vanish, or equivalently, those for which the induced current $T_\mu$ has empty boundary. Second, those that can be approximated by measures induced by embeddings of closed submanifolds (or more precisely, by parameterized cycles). Our main result, Theorem \ref{thm:holonomic}, states that these two classes coincide.  We give precise definitions in Section \ref{sec:setting}, where we also state our result. Section \ref{sec:proof} is devoted to the proof of the theorem.

Our theorem is considerably more difficult than the existing one for the case of integral currents (see for example \cite[\S4.2.9]{federer}) because we deal with arbitrary superpositions of submanifolds, and we control simultaneously the boundary, the parameterizations, and the convergence of actions of continuous Lagrangians.

In Section \ref{sec:relative}, we give the statements of similar results for manifolds and submanifolds with boundary, which can be proved using minimal modifications to the proof of Theorem \ref{thm:holonomic}.

Before plunging into the proof of the theorem, we present in Section \ref{sec:applications} some simple applications to the theory of foliations on the torus, as an illustration of what our results can be used for. Other examples of applications are given in \cite{myvariationalstructure}.

The $n=1$ case of this result was proved by Bangert \cite{bangert} and Bernard \cite{patrick}. The author saw a letter by Mather \cite{matherletter} in which an idea similar to Bangert's was sketched. Our proof of that case is different to theirs. We remark that there exist other directions in which the philosophy of these works could be generalized, such as those studied in \cite{bangertcui,bernardbessi}.

As explained in Section \ref{sec:relative}, our proof can be adapted to prove a similar statement in which the submanifolds are allowed to have a boundary contained in certain subsets of $M$.

\begin{rmk}
The two classes of measures we consider have received in the past the names \emph{closed} and \emph{holonomic}, with either term confusingly referring to either of the two classes in different parts of the literature. 
\end{rmk}

\paragraph{Acknowledgements.}

I am deeply indebted to Gonzalo Contreras for suggesting the problem treated in this paper to me, and for numerous conversations on the subject. I am also deeply indebted to Patrick Bernard, Matilde Mart\'inez, and John N. Mather for numerous discussions on this topic.

\section{Setting and statement of results}
\label{sec:setting}
\paragraph{Riemannian structure.}
Throughout, we fix a compact, oriented $C^\infty$ manifold $M$, without boundary, of dimension $d\geq 1$. Denote by $TM$ its tangent bundle, and by $T^nM$ the direct sum bundle
\[T^nM=\underbrace{TM\oplus\cdots\oplus TM}_n.\]
The dimension of $T^nM$ is $d(n+1)$. We will refer to its elements as 
\[(x,v_1,v_2,\dots,v_n),\] 
where $x\in M$ and $v_i\in T_xM$. Sometimes for brevity we will write $v$ instead of $(v_1,\dots,v_n)$. 

We will use the word \emph{smooth} to refer to $C^\infty$ functions. The space of smooth, real-valued, compactly supported functions on $T^nM$ will be denoted $\Cinfc$.

We fix a Riemannian metric $g$ on $M$, together with its Levi-Civita connection. We denote $|v|=\sqrt{g(v,v)}$ for $v\in T_xM$ and we extend this norm to $T^nM$ by letting
\[|(v_1,v_2,\dots,v_n)|=\sqrt{|v_1|^2+|v_2|^2+\cdots+|v_n|^2}.\]

Let $\vol_k(v_1,\dots,v_k)$ denote the volume of the paralellepiped spanned by the vectors $v_1,\dots,v_k\in T_xM$, as induced by the metric $g$ on $M$ by
\[\vol_k(v_1,\dots,v_k)=\left|\det(g(v_i,v_j))_{i,j=1}^n\right|.\]
Abusing notations, we will also denote $\vol_k$ the $k$-dimensional volume of piecewise-smooth subsets of $M$, which is defined as the integral of the above $\vol_k$ over any piecewise parameterization of the given subset.

We will denote by $\Omega^k(M)$ the space of smooth differential $k$-forms on $M$. We will often consider these forms as smooth functions on $T^nM$.
We also define the projection $\pi:T^nM\to M$ by 
\[\pi(x,v_1,\dots,v_n)=x.\]

\paragraph{Mild measures.}
We let $\Pn$ be the space of \emph{subvolume functions}, that is, the space of real-valued, continuous functions $f\in C^0(T^nM)$ such that 
\[
 \sup_{(x,v_1,\dots,v_n)\in T^nM}\frac{|f(x,v_1,\dots,v_n)|}{1+\vol_n(v_1,\dots,v_n)}<+\infty.
\]
Note that all differential $n$-forms on $M$ belong to $\Pn$ when regarded as functions on $T^nM$. We endow $\Pn$ with the supremum norm and its induced topology.

We define the \emph{mass} of a positive Borel measure $\mu$ to be 
\[\mass(\mu)=\int_{T^nM}\vol_n(v_1,v_2,\dots,v_n)\,d\mu(x,v_1,\dots,v_n).\]
A positive Borel measure $\mu$ on $T^nM$ is \emph{mild} if $\mass(\mu)<+\infty$. Denote by $\Mn$ the space of mild measures.

Note that for all measures in $\Mn$, the differential $n$-forms on $M$ are integrable. It follows that these measures $\mu$ induce currents $T_\mu$, that is, bounded linear functionals $\Omega^n(M)\to \R$, given by
\[\langle T_\mu,\omega\rangle=\int\omega\,d\mu.\]

The space $\Mn$ is naturally embedded in the dual space $\Pn^*$ and we endow it with the topology induced by the weak* topology on $\Pn^*$. This topology is metrizable on $\Mn$. We can give a metric by picking a sequence of functions $\{f_i\}_{i\in\N}\subset \Cinfc$ that are dense in $\Pn$ and letting 
\begin{equation}\label{eq:metricMn}
 \dist_\Mn(\mu_1,\mu_2)= |\mass(\mu_1)-\mass(\mu_2)|+\sum_{m=1}^\infty \frac{1}{2^m\sup|f_m|}\left|\int f_m d\mu_1-\int f_m d\mu_2\right|.
\end{equation}

\paragraph{Cellular complexes.}
An $n$-dimensional \emph{cell} (or \emph{$n$-cell}) $\gamma$ is a smooth map
\[\gamma:D\subseteq\R^n\to M,\]
where $D$ is a subset of $\R^n$ homeomorphic to a closed ball, \emph{together} with a choice of coordinates $t=(t_1,t_2,\dots,t_n)$ on $D$. A \emph{chain} of $n$-cells is a formal linear combination of the form
\[a_1\gamma_1+a_2\gamma_2+\cdots+a_k\gamma_k\]
for real numbers $a_1,a_2,\dots,a_k$ and $n$-cells $\gamma_1, \gamma_2, \dots, \gamma_k$. 

Let $\gamma:D\subseteq \R^n\to M$ be an $n$-cell. Denote by $d\gamma$ the differential map associating, to each element in $D$, an element in $T^nM$. Explicitly, if we have coordinates $t=(t_1,t_2,\dots,t_n)$ on $D$, then
\[d\gamma(t)=\left(\gamma(t),\frac{\partial\gamma}{\partial t_1}(t),\frac{\partial\gamma}{\partial t_2}(t),\dots,\frac{\partial\gamma}{\partial t_n}(t)\right).\]
This map depends on our choice of coordinates $t$.

To an $n$-cell $\gamma$, we associate a measure $\measure\gamma$ on $T^nM$ defined by
\[\int_{T^nM} f\,d\measure\gamma=\int_D f(d\gamma(t))\,dt,\]
where $dt=dt_1\wedge\cdots\wedge dt_n$. In other words, the measure $\measure\gamma$ is the pushforward of Lebesgue measure on $D$ under the map $d\gamma$, $\measure{\gamma}=d\gamma_*\textrm{Leb}_D$. Similarly, to a chain of $n$-cells $\alpha=\sum_{i=1}^k a_i\gamma_i$, we associate the measure $\measure\alpha$ given by 
\[\measure\alpha=\sum_{i=1}^k a_i\measure{\gamma_i}.\]
The measure $\measure\alpha$ is an element of $\Mn$. We will say that a chain $\alpha$ is a \emph{cycle} if for all forms $\omega\in\Omega^{n-1}(M)$,
\[\int_{T^nM}d\omega\,d\measure\alpha=0.\]
This is equivalent to saying that the current induced by $\measure\alpha$ has no boundary.

\begin{thm}\label{thm:holonomic}
 Assume that $1\leq n\leq d$. Let $\mu\in\Mn$ be a positive mild measure. Then the following conditions are equivalent:
 \begin{enumerate}
  \item[\holcondition] For all forms $\omega\in\Omega^{n-1}(M)$,
  \[\int_{T^nM}d\omega\,d\mu=0.\]
  \item[\cyccondition] There exists a sequence $\{\alpha_k\}_{k\in\N}$ of cycles such that the induced measures $\measure{\alpha_k}\to\mu$ as $k\to\infty$ in the topology induced by the distance \eqref{eq:metricMn}, and such that each of the measures $\measure{\alpha_k}$ is a probability.
 \end{enumerate}
 Moreover, given any positive mild measure $\mu\in\Mn$ satisfying \holcondition\ and \cyccondition, a continuous $\mu$-integrable function $L\colon T^nM\to \R$, and a positive number $\varepsilon>0$, there is a cycle $\alpha$ such that 
 \[\left|\int L\,d\mu-\int L\,d\measure{\alpha}\right|<\varepsilon.\]
\end{thm}
Most of the rest of the paper will be devoted to proving this result.
A probability measure $\mu\in\Mn$ that satisfies Conditions \holcondition\ and \cyccondition\ is said to be \emph{holonomic}. The space of all holonomic measures is convex.

\subsection{Relative holonomic measures}
\label{sec:relative}

Since our proof of Theorem \ref{thm:holonomic} relies on smooth triangulations (to be defined in Section \ref{sec:triangulations}), it is easy to modify it in order to prove

\begin{thm}\label{thm:relholonomic}
 Assume that $1\leq n\leq d$. Let $\mu\in\Mn$ and $U\subset M$ be a closed set diffeomorphic to a union of simplices of a smooth triangulation of $M$. Then the following conditions are equivalent:
 \begin{enumerate}
 \item\label{it:relholonomicity} For all forms $\omega\in \Omega^{n-1}(M)$ such that $\omega|_U=0$,
 \[\int_{T^nM}d\omega\,d\mu=0.\]
 \item There exists a sequence $\{\alpha_k\}_{k\in\N}$ of chains such that the boundaries $\partial\alpha_k$ are contained in $U$, and such that the induced measures $\measure{\alpha_k}\to\mu$ as $k\to\infty$ in the topology induced by the distance \eqref{eq:metricMn}.
 \end{enumerate}
 Moreover, given any positive mild measure $\mu\in\Mn$ satisfying item \ref{it:relholonomicity}, a continuous function $L\colon T^nM\to \R$, and a positive number $\varepsilon>0$, there is chain $\alpha$ with boundary  contained in $U$ such that 
 \[\left|\int L\,d\mu-\int L\,d\measure{\alpha}\right|<\varepsilon.\]
\end{thm}

\begin{rmk}
The boundaries $\partial \alpha_k$ can either be defined as in singular homology (see for example \cite[\S2.1]{hatcher}), or alternatively one can interpret the condition that $\partial\alpha_k$ be contained in $U$ as meaning that 
\[\int d\omega\,d\measure{\alpha_k}=0\]
for all $\omega\in\Omega^n(M)$ such that $\omega$ vanishes on $U$.
\end{rmk}

A probability measure $\mu\in\Mn$ that satisfies the conditions in Theorem \ref{thm:relholonomic} is said to be \emph{holonomic relative to $U$}. The space of all these measures is again convex.

Another variant that can be proved easily using our methods is 

\begin{thm}\label{thm:relholonomic2}
 Let $1\leq n\leq d$ and $\mu\in\Mn$. Assume that there exists an $(n-1)$-chain $\beta$ such that,  for all $\omega \in\Omega^{n-1}$,
 \[\int d\omega\,d\mu=\int \omega\,d\measure{\beta}.\]
 Assume also that the closure of the image of $\beta$ on $M$ is contained a union of $(d-1)$-dimensional simplices of a smooth triangulation of $M$. 
 Then there exists a sequence of $n$-chains $\{\alpha_k\}_{k\in\N}$ such that $\measure{\alpha_k}\to\mu$, and $\partial \alpha_k=\beta$.
 
 Moreover, given any positive mild measure $\mu\in\Mn$ satisfying the above condition, a continuous function $L\colon T^nM\to \R$, and a positive number $\varepsilon>0$, there is a chain $\alpha$ with $\partial \alpha=\beta$ such that 
 \[\left|\int L\,d\mu-\int L\,d\measure{\alpha}\right|<\varepsilon.\]
\end{thm}

\section{Applications}
\label{sec:applications}
This section presents some examples that illustrate the usefulness of Theorem \ref{thm:holonomic}. For simplicity, we do not push them to the greatest possible generality. Other applications can be found in \cite{myvariationalstructure}.

\subsection{Integrability of tangent subbundles on the torus}

\newcommand{\T}{{\mathbb T}}

Consider the case when the manifold $M$ is the $d$-dimensional torus $\T^d=\R^d/\Z^d$ with the flat metric $g$.

As a consequence of Theorem \ref{thm:holonomic}, we have a result inspired by those of Bangert-Cui \cite[Section 6]{bangertcui}:
\begin{coro}
Let $X_1,\dots,X_n$ be smooth vector fields on $\T^d$ that define a subbundle of $T\T^d$ (i.e., $X_1,\dots,X_n$ are linearly independent at each point of $\T^d$). Then there exists a foliation of $\T^d$ with $n$-dimensional leaves if, and only if, there exists a holonomic measure $\mu$ on $T^n\T^d$ supported on the points $(x,X_1(x),\dots,X_n(x))$ for $x\in \T^d$.
\end{coro}
\begin{proof}[Sketch of proof]
If we started with a foliation, we would be able to induce a measure by taking the holonomic measures induced by large pieces of the leaves and closing them up using a small amount of measure. On the other hand, if we started with a holonomic measure, we would be able to approximate it using $n$-chains that would be arbitrarily close to the leaves of a foliation.
\end{proof}


We have the following version of the Frobenius Theorem, which is an immediate consequence of Condition \holcondition\ in Theorem \ref{thm:holonomic} and integration by parts.

\begin{coro}\label{coro:frobenius}
A set of smooth vector fields $X_1,\dots, X_n$ linearly independent at each point of $\T^d$ defines a smooth foliation (i.e., the subbundle they determine is integrable) if, and only if, there exists a smooth density $\rho$ on $M$ such that for all multiindices $I$ with $n-1$ entries we have, in local coordinates $(x_1,\dots,x_d)$ on $\T^d$, 
\begin{multline}\label{eq:frobenius}
\sum_{i=1}^d\frac{\partial}{\partial x_i}(\rho \, dx_i\wedge dx_I(X_1,\dots,X_n))=
\\\divergence\begin{pmatrix}\rho\,dx_1\wedge dx_I(X_1,\dots, X_n)\\\vdots\\\rho\,dx_d\wedge dx_I(X_1,\dots, X_n)\end{pmatrix}=0,
\end{multline}
where $dx_I=dx_{i_1}\wedge\cdots \wedge dx_{i_{n-1}}$.
\end{coro}

\begin{rmk}
Corollary \ref{coro:frobenius} is a version of the Frobenius Theorem because  it relates the integrability of the subbundle to a condition on the commutators $[X_i,X_j]$ of the vector fields.

For example, in the $n=2$ case equation \eqref{eq:frobenius} easily reduces to
\begin{equation}\label{eq:frobenius2dim}
[X_1,X_2]=\frac{\divergence(\rho X_2)}\rho X_1-\frac{\divergence(\rho X_1)}\rho X_2
\end{equation}
or in the $n=3$ case we have, for $k=1,2,\dots,d$ and denoting $X_i=(X_{i1},X_{i2},\dots,X_{id})$,
\begin{multline*}
X_{1k}[X_2,X_3]+X_{2k}[X_3,X_1]+X_{3k}[X_1,X_2]=\\
\frac{\divergence(\rho(X_{3k}X_2-X_{2k}X_3))}\rho X_1 +
\frac{\divergence(\rho(X_{3k}X_1-X_{1k}X_3))}\rho X_2 +\\
\frac{\divergence(\rho(X_{1k}X_2-X_{2k}X_1))}\rho X_3.
\end{multline*}

The condition in the original Frobenius Theorem is that for all $i$ and $j$ the commutator $[X_i,X_j]$ must be in the subspace spanned by the vector fields $X_1,X_2,\dots,X_n$. Our version makes this requirement more precise because it gives a formula in terms of $\rho$ for the coefficients of $X_1,\dots,X_n$ in the linear combination corresponding to each $[X_i,X_j]$.
\end{rmk}

\subsection{Pseudoholomorphic foliations on the 4-dimensional torus}

Let $M=\mathbb T^4=\R^4/\Z^4$. The existence of pseudoholomorphic foliations on $M$ is not so well understood; as far as we know, there are only the results of Ansorge's thesis \cite{ansorge}, relying on a result of Bangert \cite{bangertexistence}. In this section we will show how to construct some of them using the results of the previous section.

Recall that an \emph{almost-complex structure} $J$ on $M$ is a smooth vector bundle isomorphism $J\colon TM\to TM$ with $J^2=-1$. 

A mapping $F\colon S\to M$ from a Riemann surface $S$ without boundary is a \emph{pseudoholomorphic curve} if $F$ is smooth and 
\[dF\circ i=J\circ dF.\]
If $M$ can be written as a disjoint union of the images $F(S)$ of pseudoholomorphic curves $F$, then the corresponding set of pseudoholomorphic curves constitutes a \emph{pseudoholomorphic foliation} of $M$. Note that this is equivalent to having a 2-dimensional foliation whose tangent subbundle in $TM$ is invariant under the action of $J$.

Let $X$ be a smooth, non-vanishing vector field on $M$. We will look for pseudoholomorphic foliations with an associated holonomic measure $\mu$ equal to 
\[\mu=\rho \,\delta_{(X, JX)},\]
where $\rho=\rho(x)$ is a smooth, nowhere-vanishing probability on $M$. In other words, we want the target space to the foliation to be spanned by $X$ and $JX$.

We choose local coordinates $(x_1,x_2,x_3,x_4)$ such that 
\[X=e_1=(1,0,0,0).\] 
Let $B=\rho J$, $B_1=Be_1$. Then the integrabilty condition \eqref{eq:frobenius2dim} becomes
\[[e_1,B_1]=\divergence(B_1)e_1,\]
which we can rewrite as
\[\frac{\partial (B_1)_2}{\partial x_2}+\frac{\partial (B_1)_3}{\partial x_3}+\frac{\partial (B_1)_4}{\partial x_4}=0,\quad \frac{\partial (B_1)_2}{\partial x_1}=\frac{\partial (B_1)_3}{\partial x_1}=\frac{\partial (B_1)_4}{\partial x_1}=0.\]

Thus, we have a $J$-pseudoholomorphic foliation of $M$ for every vector field $X$ such that the vector field $\proj_{X^\perp}(\rho JX)$ that is the projection (with respect to the standard metric induced by the local coordinates $x_1,x_2,x_3,x_4$) of $\rho JX$ onto the hyperplane perpendicular to $X$ is a vector field that is constant with respect to the flow of $X$ and has vanishing divergence. Thus one can work backwards: by choosing vector fields $X$ and $Y$ such that $\proj_{X^\perp}(Y)$ is constant with respect to the flow of $X$ and has vanishing divergence, one can then define $J$ and $\rho$ such that $Y=\rho JX$ and $J^2=-1$, to obtain a $J$-pseudoholomorphic foliation.

\section{Proof}
\label{sec:proof}
This section is devoted to the proof of Theorem \ref{thm:holonomic}, which will be given in Section \ref{sec:conclusion}.

The idea of the proof is the following. The fact that Condition \cyccondition\ implies Condition \holcondition\ is an easy consequence of Stokes's theorem, so we concentrate in the other implication. 

We start with a positive measure $\mu$ that satisfies Condition \holcondition. We prove in Section \ref{sec:smoothing} that we may assume that the measure $\mu$ is a smooth density.
In Section \ref{sec:triangulations} we specify a family of triangulations $T_k$ on $M$ for $k\in \N$.
Then in Section \ref{sec:basemeasureconstruction} we construct `base measures' $\base_k$, which are approximations to our smooth density that are (in a sense) constant on each simplex of $T_k$; this is analogous to approximating a smooth function on $\R$ with simple functions. In Section \ref{sec:basemeasureapproximation} we construct an $n$-chain $\baseapprox_k$ that is again (in a sense) almost constant on each simplex of $T_k$. This chain, however, is in general not a cycle.

In Section \ref{sec:boundary} we derive a condition on the $(d-n)$-dimensional skeleton of $T_k$ that in Section \ref{sec:cyclesconstruction} allows us to construct cycles that contain the chains $\baseapprox_k$, and whose mass $\mass$ can be estimated. We work on the estimates for the mass in Section 
\ref{sec:isoperimetric}. Finally, we put everything together in Section \ref{sec:conclusion}.


\subsection{Smoothing}
\label{sec:smoothing}
\begin{lem}\label{lem:smoothing}
Any measure $\mu$ in $\Mn$ can be approximated arbitrarily well (with respect to the metric \eqref{eq:metricMn}) using a smooth density on $T^nM$. If $\mu$ is a probability measure that satisfies Condition \holcondition\ then it can be approximated by smooth probability densities that also satisfy Condition \holcondition.
\end{lem}
\begin{proof}
Denote the exponential map by $\exp_x\colon T_xM\to M$.

A \emph{mollifier} $\psi\in C_c^\infty(\R)$ is a function such that $\psi(x)=\psi(-x)$, $\int\psi=1$, and $\psi\geq 0$. 

Fix a set of smooth vector fields $F_1,F_2,\dots,F_\ell$ on $M$ such that for each $x\in M$ the vectors $F_1(x),\dots,F_\ell(x)$ span all of $T_xM$. Note that $\ell\geq d=\dim M$. 

Denote by $\phi^i\colon M\times \R\to M$ the \emph{flow} of $F_i$:
\[\phi^i_0(x)=0,\quad\frac{d\phi^i_s(x)}{ds}=F_i(\phi^i_s(x)),\quad s\in \R.\]
For fixed $s\in \R$, denote the derivative of the diffeomorphism $\phi^i_s$ by 
\[d\phi_s^i\colon TM\to TM.\] 
Extend it to a map $d\phi^i_s\colon T^nM\to T^nM$ by setting
\[d\phi^i_s(x,v_1,v_2,\dots,v_n)=(\phi^i_s(x),d\phi^i_s v_1,\dots,d\phi^i_s v_n).\]

For $f\in \Cinfc$, we will denote by $P_i(f)$ the function given by
\[P_i(f)(x,v_1,v_2,\dots,v_n)=\int_\R f\circ d\phi_s^i(x,v_1,\dots,v_n)\,\psi(s)\,ds.\]
This is a convolution in the horizontal direction $F_i$.
Also, for $f\in \Cinfc$ we let $V(f)$ be the convolution in the vertical direction,
\begin{align*}
 V(f)(x,v_1,\dots,v_n)&=\int_{T_xM}dw_1\psi(|w_1-v_1|)\int_{T_xM}dw_2\psi(|w_2-v_2|)\\
 &\cdots \int_{T_xM} dw_n\psi(|w_n-v_n|)f(x,w_1,w_2,\dots,w_n).
\end{align*}

For $f\in \Cinfc$, we will denote
\[\psi*f=P_1P_2\cdots P_\ell V(f).\]
Note that $\psi*f$ is a $C^\infty$ function even if $f$ is only measurable. Moreover, if the diameter of the support of $\psi$ is sufficiently small, and if $f$ is an exact form on $M$, i.e.~$f(x,v_1,\dots,v_n)=d\omega_x(v_1,\dots,v_n)$ for some $\omega\in\Omega^{n-1}(M)$, then $\psi*d\omega$ is the exact form $d(\psi*\omega)$. To see this, note first that by linearity of $\omega$ on each entry $V(d\omega)=d\omega$. Also, for $s$ small enough, $\phi^*_s$ is a diffeomorphism and hence
\[P_i(d\omega)=\int \psi(s)\phi_s^{i*}d\omega\,ds=d\left[\int\psi(s)\phi_s^{i*}\omega\,ds\right]=d(P_i\omega).\]

Now let $\mu$ be a probability measure on $T^nM$. 
We define the convolution $\psi * \mu$ by duality, setting
\[\int_{T^nM} f\,d(\psi*\mu)=\int_{T^nM}(\psi*f)\,d\mu.\]
Then $\psi*\mu$ is a smooth density (see for example \cite[\S5.2]{friedlander}), and in the topology of $\Mn$,
\[\psi*\mu\to\mu\quad\textrm{as}\quad\diam\supp\psi\to0.\]
Also, if $\mu$ satisfies Condition \holcondition, then
\[\int_{T^nM}d\omega\,d(\psi*\mu)=\int_{T^nM}d(\psi*\omega)d\mu=0,\]
so $\psi*\mu$ also satisfies Condition \holcondition.
\end{proof}

\subsection{Triangulations}
\label{sec:triangulations}
A triangulation $T=(K,h)$ of $M$ is a simplicial complex $K$ homeomorphic to $M$ together with a homeomorphism $h:K\to M$. When talking about such a triangulation $T$, we will speak indistinctly of a simplex $U\subseteq K$ and of its image $h(U)\subseteq M$. In other words, we will ignore $K$ as a topological space, and we will instead think of the triangulation as being `drawn' directly on $M$. 



We say that a triangulation $T$ on the manifold $M$ is \emph{smooth} if each $d$-dimensional simplex in $T$ is the image of the standard $d$-dimensional simplex 
\[\{(x_1,\dots,x_{d+1})\in \R^{d+1}:x_i\geq0,x_1+\dots+x_{d+1}=1\}\]
under a smooth map.

We fix a sequence of smooth triangulations $\{T_k\}_{k\in\N}$ on $M$ such that:
\begin{enumerate}[T1.] 
 \item\label{it:refinement}(Successive refinements) For $k>1$, $T_k$ is a refinement of $T_{k-1}$. 
 
 For each simplex $V$ in $T_k$, $k\geq1$, we denote by $U(V)$ the simplex of dimension $d$ of $T_1$ in which $V$ is contained. (This is ambiguous for the simplices of dimension less than $d$, but any choice will work, so we assume that this choice has been made for each simplex $V$ once and for all.)

 \item\label{it:finitesimplices}(Finite) $T_k$ has finitely many simplices.
 \item\label{it:charted}(Charted) For each simplex $U$ of dimension $d$ of $T_1$, there is a chart $\varphi_U:N_U\subseteq M\to\R^d$ (for $N_U$ some neighborhood of $U$) such that the image $\varphi_U(U)$ is the standard simplex with vertices at the origin and at the vectors of the standard basis of $\R^d$.

  For brevity, we will denote $\varphi_{U(V)}$ by $\varphi_{V}$ for all simplices $V$ in the triangulations $T_k$, $k\geq1$.

 \item\label{it:linearity}(Affine) For every simplex $V$ in $T_k$, $\varphi_{V}(V)$ is contained in a translate  of a vector space $Y(V)\subset\R^d$ of dimension $\dim V$. 
 \item\label{it:nondegeneracy} (Nondegeneracy) All simplices of $T_k$ are non-degenerate. In other words, if a simplex $V$ has dimension $m$, then also
 \[\vol_mV>0.\]
 \item\label{it:vanishingdiameter}(Vanishing diameter) \[\lim_{k\to\infty}\diam T_k=0.\]
\end{enumerate}

Existence of triangulations on manifolds is discussed in great detail for example in \cite{munkres}. A triangulation $T_1$ satisfying T\ref{it:finitesimplices}--T\ref{it:nondegeneracy} 
always exists. To obtain all other refinements $T_k$ of $T_1$, one successively refines the standard simplex $\varphi_U(U)$ (for $U$ a simplex in $T_1$) making sure that the rules T\ref{it:finitesimplices}--T\ref{it:nondegeneracy} 
are respected every time. It can be seen by induction on $k$ that this is possible. One can take a refinement that respects T\ref{it:finitesimplices}--T\ref{it:nondegeneracy}. 
Ensuring overall compliance with T\ref{it:vanishingdiameter} is 
easy.  Then one pulls the resulting triangulation back to $M$ using the charts $\varphi_U$.

We will denote by $E_{m}^k$ the $m$-dimensional skeleton of the triangulation $T_k$.

\subsection{The base measure and its approximation}
\label{sec:basemeasure}
\subsubsection{Construction of the base measure}
\label{sec:basemeasureconstruction}
In Section \ref{sec:triangulations} we specified the triangulations $T_k$, $k\in\N$, and we introduced the notation $\varphi_V$.

Let $\mu$ be a smooth density in $\Mn$. We will define \emph{base measures} $0\leq \base_k\leq\mu$ depending on the triangulations $T_k$ such that $\base_k\to\mu$ as $k\to\infty$. 
Roughly speaking, the measure $\base_k$ is the largest density, constant on a constant section of $T^nM$ in the interior of each $d$-dimensional simplex $U$ of $T_k$.
Our goal here is \emph{not} to produce measures that satisfy Condition \holcondition.

For a simplex $V$ of dimension $d$ in the triangulation $T_k$, we take the chart $\varphi_V$ and extend it to a trivialization of $T^nM$, $d\varphi_V\colon T^nM\to\R^{d(n+1)}$, by setting
\[d\varphi_V(x,v_1,v_2,\dots,v_n)=\left(\varphi_V(x),d\varphi_V(v_1),\dots,d\varphi_V(v_n)\right).\]
Let $\lebesgue$ denote Lebesgue measure on $\R^{d(n+1)}$ and let $\rho$ be the Radon-Nikodym derivative of the pushforward measure $(d\varphi_V)_*\mu=\rho\lebesgue$ on $\R^{d(n+1)}$.

For $(x,v)\in \R^{d(n+1)}$ with $x\in\varphi_V(V)$, we let
\[\bar\rho_k(x,v)=\inf_{y\in\varphi_V(V)}\rho(y,v).\]
Note that $v$ is the same on both sides of the equation, and the dependence of the right-hand-side on $x$ comes from the choice of $V$. Also, this is ambiguous when $x$ lies in a simplex of dimension $<d$. This ambiguity happens only on a set of $\lebesgue$-measure zero, so we may just ignore it, as it will not affect the rest of our argument.
We let 
\[\base_k|_{T^nV}=d\varphi_V^*(\bar\rho_k\lebesgue).\]
This completely determines $\base_k$ on the whole bundle $T^nM$.
Also, $\rho_k\to\rho$ uniformly on compact sets, because $\rho$ is smooth and $\diam T_k\to0$ by T\ref{it:vanishingdiameter}. Similarly, $\mass(\base_k-\mu)\to0$. Hence $\dist_\Mn(\base_k,\mu)\to0$ as $k\to\infty$.
\subsubsection{Construction of the approximation}
\label{sec:basemeasureapproximation}
For each $k\in\N$, we will construct a chain $\baseapprox_k$ whose induced measure $\measure{\baseapprox_k}$ will approximate the base measure $\base_k$ very well. We do this in the following steps.

\paragraph{Step 1.} On each $d$-dimensional simplex $V$ of $T_k$, we sample the distribution $\bar\rho_k\lebesgue$ to get a finite sequence of points $p_1^V,\dots,p_{\ell_V}^V\in \R^{d(n+1)}$. We may assume that the following conditions are true for these points:
\begin{enumerate}[{A}1.]
 \item\label{it:pointsininterior} Each point $p_i^V$ is in the interior of $\varphi_V(V)$.
 \item\label{it:transversality} Write $p_i^V$ as $(x,v_1,\dots,v_n)\in \R^d\times\cdots\times\R^d=(\R^d)^{n+1}$. Let $\Pi^V_i$ be the plane 
 \[\Pi_{i}^V=\{x+t_1v_1+t_2v_2+\cdots+t_nv_n:t_i\in\R\}\subseteq\R^{d(n+1)}.\]
 We assume that $\Pi_i^V$ intersects all the simplices $W\subseteq \partial \varphi_V(V)$ of dimension $\dim W\geq d-n$ transversally.
 \item\label{it:weights}
 With weights $d_i^V>0$ that will be determined in Step 2, we assume that the measure 
 \begin{equation}\label{eq:deltaapprox}
\frac1{Z}\sum_{V\subset E^k_{d}}\sum_id_i^V\varphi_V^*\delta_{p_i^V},
\end{equation}
(where the sum is taken over all the $d$-dimensional simplices $V$ in the triangulation $T_k$, and $Z=\sum d^V_i$ is a normalization constant)
is a good approximation of $\base_k$, in the sense that the distance \eqref{eq:metricMn} between them tends to 0 as $k\to\infty$.
\item\label{it:densityassumption}  
 We assume that the sample $\{p_i^V\}_{i,V}$ is dense enough, in a way that will be determined in Remark \ref{rmk:pairings}.
\end{enumerate}

\paragraph{Step 2.} 
Let $V$ be a $d$-dimensional simplex in $T_k$ with respect to Lebesgue measure in the coordinates it is endowed with.
Let $\gamma_i^V\colon D_i^V\subseteq\R^n\to\R^d$ be the solution to the equations
 \begin{equation}\label{eq:diffeqforbase}
 \gamma_i^V(0,0,\dots,0)=x,\quad\frac{\partial\gamma_i^V}{\partial t_j}=v_j,\quad i=1,\dots,n.
 \end{equation}
Assume that the domain $D_{i}^V$ of $\gamma_i^V$ is the largest closed subset of $\R^n$ such that $\gamma_i^V$ remains within $\varphi_V(V)$. Note that $\image\gamma^V_i=\gamma^V_i(D^V_i)\subset \Pi^V_i$, so by A\ref{it:transversality} this image also intersects the simplices in the boundary of the standard simplex $\partial(\varphi_V(V))$ transversally.

We let $d_i^V$ be the volume $|D_i^V|$ of the domain of $\gamma^V_i$. With this definition, assumption A\ref{it:weights} can be rephrased as saying that the measure
\[\frac1{Z}\sum_{V\subset E^k_d}\sum_i \measure{(d\varphi_V)^*\gamma^V_i}\]
is a good approximation of $\base_k$. When we consider this last measure,  it is like taking the measure in equation \eqref{eq:deltaapprox}, and spreading the mass of each point along a simplex determined by its velocity vectors $v_1,\dots,v_n$. Since $\base_k|_V$ is `constant' for each such set of velocity vectors, this is in fact a very natural approximation to $\base_k$.

\paragraph{Step 3.}
Let $V$ be a simplex of dimension $d=\dim M$ in $T_k$. Let $V_1,\dots, V_\ell$ be the $d$-dimensional simplices adjacent to $V$. We may assume that each of them shares a single $(d-1)$-dimensional face $F_j$ with $V$.
For each cell $\gamma_i^V$ and each $1\leq j\leq \ell$, we try to find a cell $\gamma_m^{V_j}$ in the neighboring simplex that is almost parallel to $\gamma_i^V$ and is very close to it on the face $F_j$. 
It will not always be possible to pair up a cell in $V$ with cells in all its neighboring simplices (and in many cases it will be impossible to pair it with any), but we pair as many of them as we can. We do this on all simplices $V$ of dimension $d$. 

To make this precise, for each simplex $V$, let $I_V$ be the set of indices $i$ of the cells $\gamma^V_i$.
We say that a \emph{pairing} of the simplices is a subset $\mathsf P$ of the  union
\[\bigcup_{V,V'} I_V\times I_{V'}\]
running over all distinct $d$-dimensional simplices $V$ and $V'$ that have exactly one simplex of dimension $d-1$ in common, such that if $(i,i')\in I_V\times I_{V'}$ is in $\mathsf P$, then
\begin{itemize}
\item  $(i',i)\in \mathsf P$ as well, 
\item $(i,j)\notin \mathsf P$ for all  $i'\neq j\in I_{V'}$, and 
\item the distance on $T^nM$ of the pullbacks of the derivatives of the cells $\gamma^V_i$ and $\gamma^{V'}_{i'}$ by the charts $d\varphi_V$ and $d\varphi_{V'}$ must be close on the corresponding face $F=V\cap V'$:
\begin{multline}\label{eq:pairingcloseness}
\dist_{T^nM}\big((d\varphi_V)^* d\gamma_i^V\big(D_i^V\big)\cap T^n_{F}M,\\
(d\varphi_{V'})^*d\gamma_{i'}^{V'}\big(D_{i'}^{V'}\big)\cap T^n_{F}M\big)<(\diam T_k)^2.
\end{multline}
\end{itemize}
Here, $\dist_{T^nM}$ denotes the distance
\[\dist_{T^nM}((x,v),(x',v'))=\dist_M(x,x')+\inf_{\gamma}g_x(v,t_\gamma v'),\]
where $\dist_M$ denotes geodesic distance on $M$, the infimum is taken over all smooth curves $\gamma$ on $M$ joining $x$ and $x'$, and $t_\gamma v'$ denotes the parallel transport on $\gamma$ of $v'$ from $T^n_{x'}M$ to $T^n_xM$, done separately on each entry of $v'=(v'_1,\dots,v'_n)$ and using the Levi-Civita connection induced by the Riemannian metric $g$ on $TM$.

(Note that we are not assuming that any of these pairings will result in a closed chain, although this would indeed be the case if all the simplices were paired and glued together.)

\begin{rmk}\label{rmk:pairings}
Since $\mu$ is smooth and since $\diam T_k\to0$ in the $k\to\infty$ limit, $\base_k$ tends to be very similar on the fiber $T^n_xM$ of a point $x$ in $V$ and on the fiber of a point of an adjacent simplex $V_j$. Thus by taking $k$ large, and a sufficiently large sample $\{p^V_i\}_i$, one can pair up a proportion of the simplices that can be made arbitrarily close to being all of them. This is what we mean with assumption A\ref{it:densityassumption}: the samples must be large enough that the ratio of unpaired to paired simplices will tend to 0 as $k\to\infty$.

\end{rmk}

\paragraph{Step 4.}
Taking into account the pairing $\mathsf P$ found in Step 3, we deform the corresponding simplices ever so slightly, so that they will be glued together smoothly. We require that the first and second derivatives of the glued cells coincide throughout the gluing, which should happen precisely within the corresponding face $F=V\cap V'$. Because of condition \eqref{eq:pairingcloseness}, the necessary deformation is extremely small. We will use the same notation $\gamma^V_i$ for the deformed cells. Many of these will be identical to the original ones because they will not be paired to anything.

\paragraph{Step 5.}
We let
\[\baseapprox_k=\frac{1}{Z}\sum_{V\subset E_d^k \ell_V}\sum_i(d\varphi_V)^*\gamma^V_i.\]
  We remark that the cells $\gamma^V_i$ involved are the ones deformed as described in Step 4. The induced measure $\measure{\baseapprox_k}$ is evidently a very good approximation to $\base_k$, in the sense that their distance \eqref{eq:metricMn} vanishes asymptotically.

\subsection{Conditions on the boundary}
\label{sec:boundary}
We say that a sequence of simplices $V_1,\dots V_\ell$ of a triangulation is \emph{properly nested} if $V_i\subset \partial V_{i-1}$ and $\dim V_i=d-i$.

Let $V$ be a simplex in a triangulation $T$ of $M$. For $x$ in $V$, let
\[u_V(x)=\dist(x,\partial V).\]
If the triangulation $T$ is reasonably nice, $u_V$ can then be extended to all of $M$ in such a way that $u_V$ will be smooth on the interiors of the simplices of $\partial V$. In our case, this can be done because the triangulation satisfies T\ref{it:charted}--T\ref{it:nondegeneracy}. There is some ambiguity in the choice of the extension, but it is immaterial in our argument.

Let, for $\varepsilon>0$,
\[u^\varepsilon_V(x) = \left\{\begin{array}{ll}
u_V(x)/\varepsilon,& \textrm{if $|u_V(x)|<\varepsilon$,}\\
-1, & \textrm{if $u_V(x)<-\varepsilon$,}\\
1, &\textrm{if $u_V(x)>\varepsilon$.}
\end{array}\right.
\]
Finally, let $\bar u^\varepsilon_V$ be a smoothed version of $u^\varepsilon_V$, such that the amount of smoothing tends to 0 as $\varepsilon\to0$. This can be obtained, for example, by convolving as in Section \ref{sec:smoothing} and ensuring that one uses mollifiers $\psi$ such that $\diam\supp\psi<\varepsilon^2$.

Let $C=\{V_1\supset\cdots\supset V_n\}\subseteq T_k$ be a set of $n$ properly nested simplices. Observe that the form
\[\omega_\varepsilon=d\bar u_{V_1}^\varepsilon \wedge d\bar u_{V_2}^\varepsilon \wedge \cdots \wedge d\bar u_{V_n}^\varepsilon\]
is exact.

Let $\nu$ be a measure on $T^nM$.
Let $C=\{V_1\supset V_2\supset\cdots\supset V_\ell\}$ be properly nested simplices in some triangulation of $M$. Let 
\[B_\varepsilon(C)=\{x\in M:|u_{V_i}(x)|\leq \varepsilon,i=1,2,\dots,\ell\}. \]
Define the measure $\nu^C$ by
\begin{equation}\label{eq:slicedef}
\int f\,d\nu^C=\lim_{\varepsilon\to0+}\frac{1}{\varepsilon^{\ell}}\int_{B_\varepsilon(C)} f\,d\nu,
\end{equation}
where $f\in \Cinfc$. 

Observe that
\[\lim_{\varepsilon\to0+}\int \omega_\varepsilon\,d\mu=\int du_{V_1}\wedge du_{V_2}\wedge\cdots\wedge du_{V_n}\,d\mu^C.\]
Since the left-hand-side vanishes when $\mu$ satisfies Condition \holcondition, we get

 \begin{lem}\label{lem:boundarycondition}
 If the smooth density $\mu\in\Mn$ satisfies Condition \holcondition, then for every  $k\in\N$ and for every properly nested sequence of simplices $C=\{V_1\supset V_2\supset\cdots\supset V_{n}\}$ of the triangulation $T_k$, we have
  \begin{equation} \label{eq:boundarycondition}
   \int_{T^nM} du_{V_1}\wedge du_{V_2}\wedge\cdots\wedge du_{V_{n}}\,d\mu^C=0.
  \end{equation}
 \end{lem}
 
\begin{rmk}
 We will use Lemma \ref{lem:boundarycondition} to guide us on the `reconstruction' of the cycles that approximate $\mu$. To understand the significance of the left-hand-side of \eqref{eq:boundarycondition}, consider the case $n=1$. In this case, if we have a segment $\gamma\colon [a,b]\subset \R\to M$ with $\gamma(a)$ outside $V_1$ and $\gamma(b)$ in the interior of $V_1$, then 
 \[\int_{T^nM}du_{V_1}d\mu_\gamma^{V_1}=\int_a^bdu_{V_1}(\gamma'(t))dt=u(b)-u(a)=1-(-1)=2,\]
 while if $\gamma$ were instead going from the interior of $V_1$ to its exterior we would get $-2$. We observe that the sign gives information about whether the curve $\gamma$ is entering or exiting $V_1$, and \eqref{eq:boundarycondition} can be loosely interpreted to say that there are the same amount of curves going in as going out. In higher dimension, the story is more complicated, but the idea is the same: Lemma \ref{lem:boundarycondition} can be interpreted as a sort of perfect balance between the $n$-chains passing through the boundary of $V_1$ with each different orientation.
\end{rmk}

\subsection{Closing up the approximation to the base measure}
\label{sec:closing}
\subsubsection{Inductive construction of cycles}
\label{sec:cyclesconstruction}
In this section we inductively construct $n$-dimensional cycles $\finalchain$ that contain the chains $\baseapprox_k$ that approximate the base measure $\base_k$. Our starting point will be a measure $\measure{\finalchain^0}$ corresponding to a fictitiuos $n$-chain $\finalchain^0$ that will help us guess what the 0-dimensional intersections of $\finalchain$ with the skeleton $E^k_{d-n}$ should be.

We give the general idea in Figure \ref{img:cycles}.

\begin{figure}
  \centering
  \begin{subfigure}[b]{0.595\textwidth}
    \includegraphics[width=0.833\textwidth]{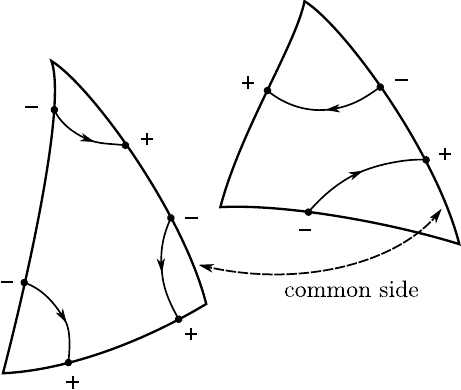}
    \caption{}
    \label{fig:1a}
  \end{subfigure}
  \begin{subfigure}[b]{0.395\textwidth}
    \includegraphics[width=0.75\textwidth]{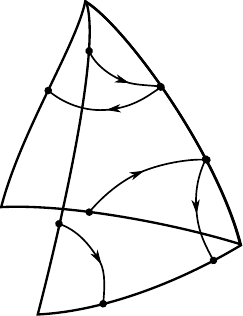}
    \caption{}
    \label{fig:1b}
  \end{subfigure}
  \begin{subfigure}[b]{0.595\textwidth}
    \centering
    \includegraphics[width=0.5\textwidth]{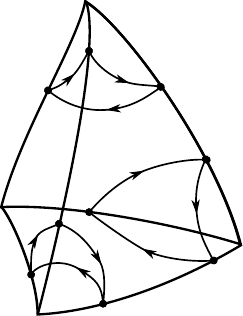}
    \caption{}
    \label{fig:1c}
  \end{subfigure}
  \begin{subfigure}[b]{0.395\textwidth}
    \includegraphics[width=0.75\textwidth]{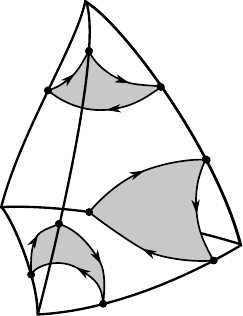}
    \caption{}
    \label{fig:1d}
  \end{subfigure}

  \caption{Construction of cycles in the case $d=3$, $n=2$. For simplicity, we ignore the pieces coming from $\baseapprox_k$.
  (a) On each simplex of dimension $d-n=1$ we have marked points with signs. (These points correspond to the projection onto $M$ of the support $\{p_i^k\}_i$ of the measure $\measure{\finalchain^0}$, and the signs are those of $W(p_i^k,C)$.) We use these signs to add 1-dimensional cells on the corresponding 2-dimensional simplex. Note that from the point of view of two 2-dimensional simplices that have an adjacent edge, the signs are the opposite. (b) Once we put the 2-dimensional simplices together, we see that the corresponding 1-chains $\finalchain^C$ fit together in a way that the corresponding chains have no boundary on the common edge. (c) The same is true for all faces of a simplex of dimension 3, so the resulting 1-chain on its boundary is itself a union of circles. (d) We can thus find a 2-chain that has the 1-chains as boundary. \\
  Similarly, in the next step we get a global 2-chain since the boundary 1-chains will cancel out on the common 2-dimensional faces of the simplices of the triangulation.
  }
  \label{img:cycles}
\end{figure}
%

\paragraph{The 0-dimensional chain.}
Recall that the chain $\baseapprox_k$ was constructed in Section \ref{sec:basemeasureapproximation}. It is a linear combination of $n$-cells $\varphi^*_V\gamma^V_i$. Although these cells originally followed equation \eqref{eq:diffeqforbase}, many of them were deformed to glue them with their paired cells, according to the pairing $\mathsf P$. We are now interested in extending the cells slightly in the directions in which they were \emph{not}  paired up with a neighboring cell. For each $k>0$, we let $\tilde\baseapprox_k$ be the chain that results from extending the domain of the $n$-cell $\gamma^V_i$ to an open set very slightly larger than its original domain $D^V_i$, so that it now intersects the skeleton $E^k_{d-1}$ of $T^k$ on the faces that it was `touching' but that corresponded to directions in which it was not paired up with anything. By property A\ref{it:transversality}, the intersection of the cell with $E^k_{d-1}$ is transversal. Then, for properly-nested simplices $C=\{V_1\supset \cdots \supset V_\ell\}$ the measure $\measure{\tilde\baseapprox_k}^C$ defined in equation \eqref{eq:slicedef} reflects the way the boundary of $\baseapprox_k$ intersects $\partial V_\ell$.

For a point $p$ in $T^nM$ such that $\pi(p)\in V_\ell$, and for a set of $n$ properly nested simplices $C=\{V_1\supset \cdots\supset V_n\}$ let 
\[W(p,C)=du_{V_1}\wedge du_{V_2}\wedge\cdots\wedge du_{V_n}(p),\]
where the functions $u_{V_i}$ are as in Section \ref{sec:boundary}.
Observe that if $C$ and $C'$ are two sets of $n$ properly nested simplices that differ only in the $\ell^{\textrm{th}}$ simplex, $\ell<n$, and the corresponding simplices $V_\ell$ and $V_\ell'$ are adjacent, then
\begin{equation}\label{eq:neighboringWs}
W(p,C)=-W(p,C')
\end{equation}
because $du_{V_\ell}=-du_{V_\ell'}$ at $p$.

For each $k$, we pick a finite set of points $\{p^k_i\}_i\subset T^nM$, and weights $r_i^k\in\R_+$ such that Conditions U\ref{U:pointsonboundary}--U\ref{U:isofrightmass} below are true. 
We want to construct a measure $\measure{\finalchain^0}$ that will capture the way in which our cycles will ultimately intersect the skeleton $E^k_{d-n}$. This measure will be the starting point for the full construction of the cycles. Crucially, at each point in its support $\measure{\finalchain^0}$ carries information about the $k$-dimensional subspace that will eventually turn out to be the intersection of our cycles $\finalchain^n$ with the skeleton $E^k_{d-n}$.
 We will imagine that there is an $n$-chain whose (degenerate) cells are the points $\{\pi(p_i)\}_i\subseteq M$, so that $\eta^0_k$ is given by
 \[\finalchain^0=\sum_ir^k_i \pi(p_i^k),\]
 and parameterized so that 
 \[\measure{\finalchain^0}=\sum_ir_i^k\delta_{p^k_i}.\]
 Strictly speaking, such a chain $\finalchain^0$ does not exist, but the measure $\measure{\finalchain^0}$ does, and this is the object we need.

 The conditions are:
 
\begin{enumerate}[U1.]

 \item \label{U:pointsonboundary}The projection $\pi(p_i^k)$ of each point $p_i^k$ on $M$ is contained in the $(d-n)$-dimensional skeleton $E^k_{d-n}$ of the triangulation $T_k$.
 
 \item \label{U:containsbase}
 We require the points in the support of $\measure{\tilde\baseapprox_k}^C$ to be contained in $\{p_i^k\}_i$, and the corresponding weights $r_i^k$ to be at least as large as the weights these points have in the measure $\measure{\tilde\baseapprox_k}^C$.

 \item \label{U:respectsboundarycondition}
 For each set of $n$ properly nested simplices $C=\{V_1\supset\cdots\supset V_{n}\}\subseteq T^k$,
 \[\sum_iW(p_i^k,C)\,r^k_i=0,\]
 where the sum is taken over all $i$ such that $\pi(p_i)$ is in $V_n$.
 
 \item \label{U:isofrightmass}
  The measure $\measure{\finalchain^0}$ approximates the restriction of $\mu$ to the skeleton $E_{d-n}^k$:
  \[\dist_{\Mn}\left(\sum_C\mu^C,\sum_C\measure{\finalchain^0}^C\right)\leq \frac{1}k\]
 where the sums are taken over all sets $C$ of $n$ properly nested simplices of $T^k$.
\end{enumerate}
The idea is that $\{p^k_i\}_i\cap \pi^{-1}(V_n)$ should be a very good sample of the 
measure 
$\mu^C$.
The set of points and weights can be found as follows. Start with the points in the support of $\measure{\tilde\baseapprox_k}^C$, with the weights they inherit from $\baseapprox_k$. Then by further sampling the measure $\mu^C$, and invoking the fact that it satisfies the conclusion of Lemma \ref{lem:boundarycondition}, a solution for the condition in item U\ref{U:respectsboundarycondition} is guaranteed to exist. Note that the condition in item U\ref{U:respectsboundarycondition} is essentially a rephrasing of the conclusion of Lemma \ref{lem:boundarycondition} adapted to $\measure{\finalchain^0}^C$. Taking a sufficiently large sample of $\mu^C
$, one can also guarantee that item U\ref{U:isofrightmass} will be satisfied.

\paragraph{The higher-dimensional chains.}

For every set of $n+1$ properly nested simplices $C=\{V_1\supset\cdots\supset V_{n+1}\}$, we let $\finalchain^{C}$ denote the 0-dimensional chain 
\[ \finalchain^C=\sum_i (\sgn W(p_i^k,C))r_i^k\pi(p_i^k)\]
where the sum is taken over all indices $i$ such that $p_i^k$ is contained in $V_{n+1}$. 

For every set $C=\{V_1\supset\cdots \supset V_{n-j}\}\subseteq T_k$ of $n-j$ properly nested simplices, $1\leq j< n$, $\tilde\baseapprox_k$ induces an $j$-dimensional chain $\baseapprox^C_k$ on $\partial V_{n-j}$ that satisfies, for all $\omega\in \Omega^{j}(M)$,
 \[
 \int_{\baseapprox^C_k} \omega=
 \int_{T^nM} \omega\wedge du_{V_1}\wedge du_{V_2}\wedge \cdots\wedge du_{V_{n-j}} \, d\measure{\tilde\baseapprox_k}
 \]
Observe that the chain $\baseapprox_k^C$ is in general not unique, but any choice will do for our purposes. We also let $\baseapprox_k^\emptyset=\baseapprox_k$. 

For sets of properly nested simplices 
\[C'=\{V_1\supset \cdots \supset V_{n-j-1}\}\subset C=\{V_1\supset\cdots \supset V_{n-j}\},\] 
we refine the chain $\baseapprox_k^{C'}$ so that each of its $(j+1)$-dimensional cells intersects only one of the $(d-n+j+1)$-dimensional simplices of the boundary $\partial V_{n-j-1}$.
We then let $\bar\baseapprox_k^{C}$ be the part of $\baseapprox_k^{C'}$ that is contained in $V_{n-j}$. In other words,
\[\baseapprox^{C'}_k=\sum_{V\subset\partial V_{n-j}}\bar\baseapprox_k^{C'\cup\{V\}}.\]

We proceed to  construct, inductively on $j=0,1,\dots,n-1$, $(j+1)$-dimensional cycles $\finalchain^{C}$ corresponding to each set of $n-j$ properly nested simplices $C=\{V_1\supset\cdots\supset V_{n-j}\}\subseteq T_k$, such that:
\begin{enumerate}[E1.]
 
 \item \label{E:projectioninskeleton}  The cells of $\finalchain^C$ are contained in $
 V_{n-j}\subseteq E_{d-n+j+1}^k\subseteq M$.
 
  \item \label{E:containsbaseapprox}
 We require that $\bar\baseapprox_k^C$ be contained in $\finalchain^C$, in the sense that all the cells of $\bar\baseapprox_k^C$ appear in $\finalchain^C$ with coefficients of magnitude greater or equal to those they have in $\bar\baseapprox_k^C$. 
 
 If $j=n-1$, $C=\{V_1\}$ and  $\finalchain^C$ contains precisely the cells of $\baseapprox_k$ that are contained in $V_1$, and with exactly the same parameterization for each cell.
 
 \item \label{E:boundaryinheritance} We have
 \[\partial\finalchain^{C} = \sum_{V\subset \partial V_{n-j}}\finalchain^{C\cup \{V\}},\]
 where the sum is taken over all simplices in the boundary of $V_{n-j}$. 
 
 \item \label{E:adjacencycondition}If $C$ and $C'$ are sets of $n-j$ properly nested simplices of $T_k$ that only differ in the $\ell$-th simplex, $1\leq\ell< n-j$, and the corresponding simplices $V_\ell$ and $V_\ell'$ are adjacent, then
 \[\finalchain^C=-\finalchain^{C'}.\]
 This should hold in the sense that the induced functionals on $\Omega^{j+1}(M)$ (i.e., the induced currents) must be equal.
 
 \item \label{E:iscycle} If  $C'=\{V_1\supset\cdots\supset V_{n-j-1}\}\subseteq T_k$ is not empty, 
 \[\sum_{V\subset\partial V_{n-j-1}}\partial\finalchain^{C'\cup \{V\}}=0,\]
 where the sum is taken over all simplices in the boundary of $V_{n-j-1}$. If $C'$ is empty, then the same equation should hold, but now taking the sum over all simplices $V$ of dimension $d$ in $T_k$. 
 
 \item \label{E:massminimizing}The cells of $\finalchain^{C}$ that are not inherited from $\bar\baseapprox_k^C$ are almost $\mass$-mass minimizing, in a sense that will be specified at the end of Section \ref{sec:isoperimetric}. 
 
 
 
 
\end{enumerate}

First we show how to create the 1-chain $\finalchain^C$ corresponding to the case in which $C$ contains $n$ properly nested simplices. We start with $\bar\baseapprox_k^C$, which will provide for compliance with item E\ref{E:containsbaseapprox}. By U\ref{U:containsbase}, the boundary of $\bar\baseapprox_k^C$ is also contained in $\sum_{V\subset \partial V_{n-1}}\finalchain^{C\cup \{V\}}$. So what we do, in order to comply with E\ref{E:projectioninskeleton} and E\ref{E:boundaryinheritance}, is that we connect the remaining dots in $\sum_{V\subset \partial V_{n-1}}\finalchain^{C\cup \{V\}}$ with curves contained in $V_{n-1}$ in the way prescribed by the weights of the dots; because of property U\ref{U:respectsboundarycondition}, this is possible. By taking very short curves, we ensure compliace with E\ref{E:massminimizing}. Because of identity \eqref{eq:neighboringWs}, the construction of $\finalchain^{C\cup \{V\}}$ ($V\subset\partial V_{n-1}$) immediately implies E\ref{E:adjacencycondition}. Property E\ref{E:iscycle} also follows from the identity \eqref{eq:neighboringWs}.

Now assume that we have $\finalchain^C$ for $j=m-1$, and let us construct it for $j=m$, $m> 1$. Let $C=\{V_1\supset\cdots\supset V_{n-m}\}\subseteq T_k$. For each simplex $V\subset \partial V_{n-m}$, we are assuming that there exists $\finalchain^{C\cup V}$ that satisfies E\ref{E:projectioninskeleton}--E\ref{E:massminimizing}. To close these up, we again start with $\bar\beta_C$ (whence complying with E\ref{E:containsbaseapprox}) and we add cells of dimension $m+1$ contained in $V_{n-m}$ (complying with E\ref{E:projectioninskeleton}) so that property E\ref{E:boundaryinheritance} will hold; this is possible because $V_{n-m}$ has trivial homology and because $\sum_{V\subset \partial V_{n-1}}\finalchain^{C\cup \{V\}}$ is a cycle as it satisfies E\ref{E:iscycle}. Properties E\ref{E:adjacencycondition} and E\ref{E:iscycle} for $j=m$ follow from property E\ref{E:adjacencycondition} for $j=m-1$. Compliance with property E\ref{E:massminimizing} can be attained by choosing an almost mass-minimizing set of $(m+1)$-cells.

Write $\finalchain=\finalchain^\emptyset$.
We have proved:
\begin{lem}
There is a sequence of cycles $\finalchain$ that contain $\baseapprox_k$ and such that
\begin{equation}\label{eq:massdifference}
\mass(\measure{\finalchain})-\mass(\measure{\baseapprox_k})
\end{equation}
is almost minimal (in the sense of E\ref{E:massminimizing}), while respecting
\begin{equation}\label{eq:zeroskeletonmeasurecomparison}
\dist_{\Mn}\left(\sum_C\mu^C,\sum_C\measure{\finalchain}^C\right)\leq \frac{1}k,
\end{equation}
where the sums are taken over all sets $C$ of $n$ properly nested simplices of $T^k$. Also, the part of $\measure{\finalchain}^C$ that comes inherited from $\baseapprox_k$ satisfies A\ref{it:densityassumption}.
\end{lem}
By construction, equation \eqref{eq:zeroskeletonmeasurecomparison} is exactly the same as the condition in U\ref{U:isofrightmass}. 

\subsubsection{Isoperimetric inequality}
\label{sec:isoperimetric}
In this section we want to find an upper bound for the mass difference \eqref{eq:massdifference}.

Recall the isoperimetric inequality:
\begin{prop}[{Federer \cite[\S4.2.10]{federer}, \cite[\S5.3]{morgan}}]
\label{prop:isoperimetric}
There is a constant $\constisoperimetric>1$ such that if $\theta$ is an $m$-chain with $\partial\theta=0$ and contained in a simplex $V$ of some triangulation $T_k$ and of diameter $\diam_VV<1$, then there exists an $(m+1)$-chain $\sigma$ with $\partial\sigma=\theta$ contained in $V$ and with mass bounded by 
\[\mass(\measure{\sigma})\leq\constisoperimetric\mass(\measure{\theta})^{\frac{k+1}k}.\]
\end{prop}
The original proposition is valid for chains $\theta$ in $\R^d$. It is true as stated because when we pullback a chain from $\R^d$ to $M$ via any of the functions $\varphi_V$, the modulus of continuity of these mappings is globally bounded. This in turn is true because there are only finitely many of them, and they have compact domains.

Let $k\ge1$ and let $V_1$ be a $d$-dimensional simplex in $T_k$. 
Let $k\geq 1$ and let $C$ be a set of properly nested simplices in $T_k$. Decompose the chain $\finalchain^{C}$ into the part of it that comes from $\bar\baseapprox_k^{C}$ and a remainder $\remainder^C_k$,
\[\finalchain^{C}=\bar\baseapprox_k^{C}+\remainder^{C}_k.\]
It follows 
Proposition \ref{prop:isoperimetric} that we can take the cells in $\remainder^C_k$ to be such that, as $k\to\infty$,
\begin{align*}
\mass(\measure{\remainder}^{\{V_1\}}_k)&\leq \constisoperimetric  \sum_{V_2\subset \partial V_1}\mass(\measure{\remainder}^{\{V_1,V_2\}}_k)^2 +\varepsilon^k_2 \\
 &\leq \constisoperimetric^{1+\frac32}  \sum_{V_3\subset \partial V_2}\sum_{V_2\subset \partial V_1}\mass(\measure{\remainder}^{{\{V_1,V_2,V_3\}}}_k)^3 +\varepsilon^k_3
 \\
 &\leq\cdots\leq \constisoperimetric^{q_n} \!\!\! \sum_{V_{n-1}\subset \partial V_{n-2}}\!\!\cdots\sum_{V_2\subset \partial V_1}\mass(\measure{\remainder}^{{\{V_1,V_2,\dots,V_{n-1}\}}}_k)^{n-1}+\varepsilon_{n-1}^{k}\\&\hspace{9.7cm}\to0,
\end{align*}
where $q_n>1$ is some number depending only on $n$, $\varepsilon_k^\ell$ is arbitrarily small (it is the error we may get from not taking exactly the cell provided by Proposition \ref{prop:isoperimetric}, but one with slightly larger mass; we thus specify property E\ref{E:massminimizing} to mean that $\varepsilon^k_\ell\to0$ as $k\to\infty$ for all $\ell$), and the sums are taken over all simplices in the corresponding boundaries. The asymptotic vanishing of the last sum follows from assumptions U\ref{U:containsbase} and U\ref{U:isofrightmass}, from Remark \ref{rmk:pairings}, and from inequality \eqref{eq:zeroskeletonmeasurecomparison}. We conclude 
\begin{lem}\label{lem:massdifferencevanishes}
\[|\mass(\measure{\finalchain})-\mass(\measure{\baseapprox_k})|\to 0\quad\textrm{as}\quad k\to \infty.\]
\end{lem}
\begin{rmk}\label{rmk:measurevanishes}
We may assume that as $k\to\infty$
\begin{equation}\label{eq:smallmeasure}
(\measure{\finalchain}-\measure{\baseapprox_k})(T^nM)\to 0,
\end{equation}
and that the support of the measure $\measure{\finalchain}-\measure{\baseapprox_k}$ is contained in a compact subset of $T^nM$ that does not depend on $k$.

Indeed, the difference $\finalchain-\baseapprox_k$ corresponds exactly to the cells we added in order to close up the chain $\baseapprox_k$ and get a cycle. We may reparameterize these cells $\gamma$ so that the measure they contribute, $\measure\gamma(T^nM)$, will be approximately equal to their mass $\mass(\measure\gamma)$, which is bounded by Lemma \ref{lem:massdifferencevanishes}. In doing so, by making sure that the partial derivaties of $\gamma$ stay almost perpendicular, we may keep $\supp\measure\gamma$ within the compact set $\{(x,v_1,\dots,v_n)\in T^nM:|(v_1,\dots,v_n)|\leq 2\}$. 
\end{rmk}

\subsection{Conclusion}
\label{sec:conclusion}
%
%

\begin{proof}[Proof of Theorem \ref{thm:holonomic}]
Let $\mu\in\Mn$ be a positive measure. If $\mu$ satisfies Condition \cyccondition, it follows from Stokes's theorem that it also satisfies Condition \holcondition.

To prove the other direction, assume that $\mu$ satisfies Condition \holcondition. 
By Lemma \ref{lem:smoothing}, we can assume that $\mu$ is smooth. We can thus construct for $k\geq 1$ triangulations $T_k$ as in Section \ref{sec:triangulations}, base measures $\base_k$ as in Section \ref{sec:basemeasureconstruction}, chains $\baseapprox_k$ approximating these as in Section \ref{sec:basemeasureapproximation}, and cycles $\finalchain$ as in Section \ref{sec:cyclesconstruction} that contain $\baseapprox_k$. We have
\begin{equation}\label{eq:finalestimate}
\dist_\Mn(\mu,\measure{\finalchain})\leq\dist_\Mn(\mu,\base_k)+\dist_\Mn(\base_k,\measure{\baseapprox_k})+\dist_\Mn(\measure{\baseapprox_k},\measure{\finalchain}).
\end{equation}
The first two summands on the right-hand-side vanish asymptotically by construction. The last term, as per the definition of $\dist_\Mn$ in equation \eqref{eq:metricMn}, has two parts: the mass difference, which tends to zero by Lemma \ref{lem:massdifferencevanishes}, and the one involving the functions $f_i$. The second one also vanishes asymptotically because the difference between $\baseapprox_k$ and $\finalchain$ corresponds to the cells added to close $\baseapprox_k$ up, and the measure these cells contribute can be taken to tend to zero, as explained in Remark \ref{rmk:measurevanishes}. Since each term in the second part of the definition \eqref{eq:metricMn} of $\dist_\Mn$ is essentially the difference of integrals of functions that are everywhere $\leq 2^{-m}$ and since the total measure involved $\measure{\finalchain}-\measure{\baseapprox_k}$ vanishes asymptotically as $k\to\infty$, the sum also vanishes in that limit.
We conclude that the measures induced by the cycles $\finalchain$ indeed approximate $\mu$, so $\mu$ satisfies Condition \cyccondition.

The last statement of the theorem follows from formula \eqref{eq:finalestimate} together with the following considerations. First, it is clear that
\[\int L\,d\mu-\int L\,d\base_k \to 0\quad\textrm{and} \int L\,d\base_k-\int L\,d\measure{\baseapprox_k}\to 0\]
because this is by construction true in any compact subset of $T^nM$ and because of the $\mu$-integrability of $L$. The difference
\[\int L\,d\measure{\baseapprox_k} - \int L\,d\measure{\finalchain}\]
also tends to zero because of equation \eqref{eq:smallmeasure} and because the support of the difference can be taken to be contained in a compact set independent of $k$, as explained in Remark \ref{rmk:measurevanishes}, and $L$ is bounded within this compact set because it is continuous.
\end{proof}

 \bibliography{bib}{}
 \bibliographystyle{plain}
\end{document}